\documentclass[12pt]{article}
\usepackage{amsmath}
\usepackage{amsfonts}
\usepackage{amsthm}
\newtheorem{theorem}{Theorem}
\newtheorem{lemma}{Lemma}
\numberwithin{equation}{section}

\title{Unique special solution for discrete Painlev\'e II}
\author{Walter Van Assche\thanks{This work is supported by FWO grant G0C9819N.} \\ KU Leuven, Belgium}
\date{\today}

\begin{document}
\maketitle

\begin{abstract}
We  show that the discrete Painlev\'e II equation with starting value
$a_{-1}=-1$ has a unique solution for which $-1 < a_n < 1$ for every $n \geq 0$. This solution corresponds
to the Verblunsky coefficients of a family of orthogonal polynomials on the unit circle. This result was already
proved for certain values of the parameter in the equation and recently a full proof was given by Duits and
Holcomb \cite{DuitsHolcomb}. In the present paper we give a different proof that is based on an idea put forward by Tomas Lasic Latimer \cite{Latimer} which uses orthogonal polynomials. We also give an upper bound for this special solution.
\end{abstract}

\section{Introduction}

The unicity of a positive solution of discrete Painlev\'e equations  has been investigated before.
\begin{itemize}
\item The discrete Painlev\'e I equation
\[    x_n(x_{n+1}+x_n+x_{n-1} +K) = \alpha n, \qquad \alpha > 0  \]
with $x_0=0$, has a unique positive solution (Lew and Quarles \cite{LewQuarles}, Nevai \cite{Nevai}). This was extended
to a more general recurrence relation by Alsulami et al. \cite{ANSVA}.
\item There is a unique  solution for the alternative discrete Painlev\'e I equation
\[    \begin{cases}
             a_n+a_{n+1}=b_n^2-t, \\
             a_n(b_n+b_{n-1}) = n,
        \end{cases}   \]
with $t>0$ and $a_0=0$, for which $a_{n+1} >0$ and $b_n>0$ for all $n \geq 0$ (Clarkson et al. \cite{CLVA}).
\item The $q$-discrete Painlev\'e equation
\[   a_n(a_{n+1}+q^{1-n}a_n+q^2a_{n-1}+q^{-2n+3}a_{n+1}a_na_{n-1}) = (1-q^n) q^{n-1}  \]
with $0 < q < 1$ and $a_0=0$, has a unique positive solution (Latimer \cite{Latimer}, where a somewhat more general
equation was investigated).  
\end{itemize}
The discrete Painlev\'e II equation that we will consider in this paper is
\begin{equation}  \label{dPII}
    x_{n+1}+x_{n-1} = \frac{\alpha  n x_n}{1-x_n^2},  
\end{equation}    
where $\alpha \in \mathbb{R}$. This is a one parameter version of d-P$(A_3^{(1)},D_5^{(1)})$ with symmetry
$A_3^{(1)}$ and surface type $D_5^{(1)}$ in the geometric classification \cite{KNY}.
We will show that there is a unique solution with $x_0=1$ for which $-1 < x_n < 1$ for
$n \geq 1$. In a similar way there is also a unique solution with $x_0=-1$ for which $-1 < x_n < 1$ for all $n \geq 1$,
which follows by taking $-x_n$ in the previous problem. This result was proved for $\alpha >1$ and $\alpha < -1$ 
in \cite[Thm. 3.10 and Cor. 3.12]{WVA}. The proof is based on a fixed point argument but this does not work for $|\alpha| \leq 1$. Recently Duits and Holcomb \cite[Thm. 2.1]{DuitsHolcomb} proved the unicity for all real $\alpha \neq 0$. Their proof used
a convexity argument showing that a related quantity has a unique minimum.
We will prove the unicity using a different approach, which was used by Latimer in \cite{Latimer} for
proving the uniqueness of a $q$-discrete Painlev\'e equation. His appoach uses orthogonal polynomials on the real line,
but we will use orthogonal polynomials on the unit circle. We will use a different version of \eqref{dPII}, with a shift 
$x_{n+1}=a_n$ and $\alpha=-2/t$. 
This corresponds better to the framework of orthogonal polynomials on the unit circle. 
Our main result is:
 
\begin{theorem}   \label{thm1}
The discrete Painlev\'e II equation
\[    a_{n+1}+a_{n-1} = \frac{-2(n+1)a_n}{t(1-a_n^2)}  \]
with initial value $a_{-1}=-1$ has a unique solution for which $-1 < a_n < 1$ for all $n \geq 0$. This
solution is obtained by taking $a_0 = I_1(t)/I_0(t)$, where $I_n(t)$ is the modified Bessel function of order $n$.
\end{theorem}

This unique solution corresponds to a special function solution in terms of modified Bessel functions,
see e.g., \cite{RGTT}. The sequence $(a_n)_{n \in \mathbb{N}}$ contains the Verblunsky coefficients of
a family of orthogonal polynomials on the unit circle which was first investigated by Periwal and Shevitz \cite{PeriwalShevitz} and
which Ismail  calls the modified Bessel polynomials \cite[pp. 236--239]{Ismail}.
The proof of Theorem \ref{thm1} is given in Section 3 and depends on two results for the orthogonal polynomials
on the unit circle with Verblunsky coefficients given by the required solution $(a_n)_n$. These results (Lemmas \ref{lem1} and \ref{lem2}) are proved in Section 2 which also contains the necessary background about orthogonal polynomials
on the unit circle. 
In Section 4 we give an upper bound for this unique solution, which shows that it tends to zero very fast.

\section{Orthogonal polynomials on the unit circle}

We will use some basic notions from the theory of orthogonal polynomials on the unit circle (OPUC). We recommend
the books of Szeg\H{o} \cite[Chapter XI]{Szego} and Simon \cite{Simon} for the general theory.
Suppose $\nu$ is a positive measure on $[0,2\pi]$ and consider the orthonormal polynomials $\{ \varphi_n(z), n=0,1,2,\ldots\}$
on the unit circle for this measure $\nu$, for which
\[   \int_0^{2\pi} \varphi_n(z)\overline{\varphi_m(z)}\, d\nu(\theta) = \delta_{m,n}, \qquad z=e^{i\theta}, \]
with $\varphi_n(z)=\kappa_n z^n + \cdots$ and $\kappa_n>0$. We will denote the monic orthogonal polynomials
by $\Phi_n(z)=\varphi_n(z)/\kappa_n$. These monic polynomials can be computed using the Szeg\H{o}-Levinson
recurrence
\begin{equation}   \label{Srec}
   z\Phi_n(z) = \Phi_{n+1}(z)+ \overline{a_n} \Phi_n^*(z),
\end{equation}   
where $\Phi_n^*(z) = z^n \overline{\Phi}_n(1/z)$ is the reversed polynomial, with $\overline{\Phi}_n$ the polynomial
with complex conjugated coefficients. The parameters $(a_n)_{n \in \mathbb{N}}$ are called Verblunsky coefficients.
They are given by $a_n = - \overline{\Phi_{n+1}(0)}$ and they satisfy $|a_n| <1$ for $n \geq 0$, and $a_{-1}=-1$.
Verblunsky's theorem (or the spectral theorem for OPUC \cite[\S 1.7 and \S 3.1]{Simon}) says that for any sequence of complex numbers
$a_n \in \mathbb{C}$ for which $|a_n| < 1$ for $n \geq 0$, there is a unique probability measure $\nu$ on the unit circle such that the polynomials $\Phi_n$ generated by the recurrence \eqref{Srec} are the monic orthogonal polynomials for this measure
$\nu$:
\[   \int_0^{2\pi} \Phi_n(z) \overline{\Phi_m(z)}\, d\nu(\theta) = \frac{\delta_{m,n}}{\kappa_n^2},  \qquad z=e^{i\theta}.   \]
An important relation between $\kappa_n$ and $a_n$ was found by Szeg\H{o} (see \cite[Eq. (11.3.6)]{Szego}):
\[    \kappa_n^2 = \sum_{k=0}^n |\varphi_k(0)|^2, \]
from which we find that $\kappa_{n+1}^2-\kappa_n^2 = |\varphi_{n+1}(0)|^2 = \kappa_{n+1}^2 |a_n|^2$, so that
\begin{equation}   \label{kapparat}
  \frac{\kappa_n^2}{\kappa_{n+1}^2} = 1-|a_n|^2.
\end{equation}
We will be using real coefficients $a_n$ for which $-1 < a_n < 1$.

\begin{lemma}  \label{lem1}
Let $(a_n)_{n \geq -1}$ with  $a_{-1}=-1$ be a solution of 
\begin{equation}      \label{dPrec}
-\frac{t}{2}(1-a_n^2)(a_{n+1}+a_{n-1}) = (n+1)a_{n},
\end{equation}
for which $-1 < a_n <1$ for all $n \geq 0$, and let $\{ \Phi_n(z), n=0,1,2,\ldots\}$ be the monic orthogonal
polynomals on the unit circle for which
\begin{equation}  \label{Srec2}
      z\Phi_n(z) = \Phi_{n+1}(z)+ a_n \Phi_n^*(z),
\end{equation}
holds. Then
\begin{equation}  \label{Phider}
    \Phi_n'(z) = n \Phi_{n-1}(z) + B_n \Phi_{n-2}(z), \qquad n \geq 1,
 \end{equation}
 with 
 \[  B_n = \frac{t}{2} \frac{\kappa_{n-2}^2}{\kappa_n^2} = \frac{t}{2}(1-a_{n-2}^2)(1-a_{n-1}^2).  \]
 \end{lemma}    

\begin{proof}
The recurrence \eqref{Phider} is clearly true for $n=1$. For $n=2$ we have  $\Phi_2(z) = z\Phi_1(z)-a_1\Phi_1^*(z)$
and $\Phi_1(z)=z-a_0$, so that
\[    \Phi_2(z) = z^2 - a_0(1-a_1)z-a_1.  \]
Then on one hand we find $\Phi_2'(z) = 2z -a_0(1-a_1)$ and on the other hand
\[    2\Phi_1(z)+B_2\Phi_0(z) = 2(z-a_0) + \frac{t}{2}(1-a_0^2)(1-a_1^2). \]
So \eqref{Phider} holds for $n=2$ if we can show that
\begin{equation} \label{aux}
    a_0(1-a_1)  - 2a_0 + \frac{t}{2}(1-a_0^2)(1-a_1^2) = 0.  
\end{equation}    
Use \eqref{dPrec} with $n=0$ and $a_{-1}=-1$ to find that $2a_0=t(1-a_0^2)(1-a_1)$, then \eqref{aux} easily follows.

We will prove \eqref{Phider} by induction. We can find  $\Phi_{n+1}(z)$ using the recurrence \eqref{Srec}
\[   \Phi_{n+1}(z) = z \Phi_n(z) - a_n \Phi_n^*(z).  \]
In order to expand $\Phi_n^*(z)$ in the basis $\{ \Phi_k(z), 0 \leq k \leq n\}$ we use the relation \cite[Eq. (11.3.5)]{Szego}
\[   \kappa_n \varphi_n^*(z) = \sum_{k=0}^n \varphi_k(z) \overline{\varphi_k(0)}, \]
which for the monic polynomials becomes
\[   \kappa_n^2 \Phi_n^*(z) = \sum_{k=0}^n \kappa_k^2 \Phi_k(z) \overline{\Phi_k(0)}.  \]
Recall that $a_k=-\overline{\Phi_{k+1}(0)}$, so that we find
\begin{equation}  \label{Phi*}
        \Phi_n^*(z) = -\sum_{k=0}^n \frac{\kappa_k^2}{\kappa_n^2} a_{k-1} \Phi_k(z).  
\end{equation}
This gives
\[   \Phi_{n+1}'(z) = z \Phi_n'(z) + \Phi_n(z) + a_n \sum_{k=0}^n \frac{\kappa_k^2}{\kappa_n^2} a_{k-1} \Phi_k'(z).  \]
The induction hypothesis allows us to write $\Phi_k'(z)$ for $0 \leq k \leq n$ in terms of $\Phi_{k-1}$ and $\Phi_{k-2}$
using \eqref{Phider}. This gives
\[   \Phi_{n+1}'(z) = z [ n \Phi_{n-1}(z) + B_n \Phi_{n-2}(z)] + \Phi_n(z) 
                          + a_n \sum_{k=0}^n \frac{\kappa_k^2}{\kappa_n^2} a_{k-1} [k \Phi_{k-1}(z) + B_k \Phi_{k-2}(z)]. \]  
We can use \eqref{Srec} to replace $z\Phi_{n-1}(z)$ and $z\Phi_{n-2}(z)$, to find
\begin{multline}   \label{Phin+1}
  \Phi_{n+1}'(z) = (n+1)\Phi_n(z) -n a_{n-1} \sum_{k=0}^{n-1} \frac{\kappa_k^2}{\kappa_{n-1}^2}a_{k-1} \Phi_k(z)
      + B_n \Phi_{n-1}(z) \\
  - B_n a_{n-2} \sum_{k=0}^{n-2} \frac{\kappa_k^2}{\kappa_{n-2}^2} a_{k-1} \Phi_k(z) 
      +  a_n \sum_{k=0}^n \frac{\kappa_k^2}{\kappa_n^2} k a_{k-1} \Phi_{k-1}(z)
      + a_n \sum_{k=0}^n \frac{\kappa_k^2}{\kappa_n^2} a_{k-1}B_k \Phi_{k-2}(z).
\end{multline}                              
We will now check the coefficients of $\Phi_k(z)$ on the right hand side of \eqref{Phin+1} to show that \eqref{Phider}
indeed holds for $n+1$. The coefficient of $\Phi_n(z)$ is $n+1$, which agrees with \eqref{Phider}. The coefficient of $\Phi_{n-1}(z)$ is
\[    -na_{n-1}a_{n-2} + B_n + na_na_{n-1} = na_{n-1}(a_n-a_{n-2}) + \frac{t}{2}(1-a_{n-2}^2)(1-a_{n-1}^2). \]
If we use \eqref{dPrec} for $n-1$ then this is
\begin{multline*}
      - \frac{t}{2}(1-a_{n-1}^2)(a_n+a_{n-2})(a_n-a_{n-2}) + \frac{t}{2}(1-a_{n-2}^2)(1-a_{n-1}^2) \\
    = \frac{t}{2}(1-a_{n-1}^2)(1-a_n^2) = B_{n+1}  ,  
\end{multline*}
which agrees with \eqref{Phider}.   
The coefficient of $\Phi_k(z)$ in \eqref{Phin+1} for $k \leq n-2$ is given by
\[     \frac{\kappa_k^2}{\kappa_n^2} \Bigl[ -na_{n-1}a_{k-1} \frac{\kappa_n^2}{\kappa_{n-1}^2} -B_na_{n-2}a_{k-1} \frac{\kappa_n^2}{\kappa_{n-2}^2} 
    + (k+1) a_na_k \frac{\kappa_{k+1}^2}{\kappa_k^2} + a_n a_{k+1} B_{k+2} \frac{\kappa_{k+2}^2}{\kappa_k^2} \Bigr]. \]
The expression between the square brackets simplifies to
\[  -a_{k-1} \left(\frac{na_{n-1}}{1-a_{n-1}^2} + \frac{t}{2} a_{n-2} \right) +
    a_n \left( \frac{(k+1)a_k}{1-a_k^2} + \frac{t}{2} a_{k+1} \right).  \]    
Using \eqref{dPrec} with $n-1$ in the first term and with $n=k$ in the second term then gives
\[  -a_{k-1}\left(  -\frac{t}{2} (a_n+a_{n-2}) + \frac{t}{2} a_{n-2} \right) 
    + a_n \left( -\frac{t}{2} (a_{k+1}+a_{k-1} ) + \frac{t}{2} a_{k+1} \right) = 0, \]
which agrees with \eqref{Phider}.        
\end{proof}

Let us denote the moments of the orthogonality measure $\nu$ for these orthogonal polynomials by
\[    \mu_n = \int_0^{2\pi} z^n \, d\nu(\theta), \qquad   z= e^{i\theta}. \]
We will normalize the orthogonality measure to be a probability measure, so that $\mu_0=1$. 
These moments satisfy a three term recurrence relation:

\begin{lemma}  \label{lem2}
Let $(\mu_n)_{n \geq 0}$ be the moments of the measure $\nu$ for the orthogonal polynomials $(\Phi_n)_{n \geq 0}$.
Then
\begin{equation}  \label{murec}
   (n-1)\mu_{n-1} + \frac{t}{2} (\mu_n-\mu_{n-2}) = 0.
\end{equation}
\end{lemma}

\begin{proof}
Write $z^n$ in the basis $\{\Phi_k(z), 0 \leq k \leq n\}$
\[   z^n = \sum_{k=0}^n c_{k,n} \Phi_k(z), \]
with $c_{n,n}=1$, then 
\begin{equation}  \label{ckn}
  c_{k,n} = \kappa_k^2 \int_0^{2\pi} z^n \overline{\Phi_k(z)} \, d\nu(\theta).  
\end{equation}  
We then see that the orthogonality implies
\[   \mu_n = \sum_{k=0}^n c_{k,n} \int_0^{2\pi} \Phi_k(z) \, d\nu(\theta) = c_{0,n}, \]
where we used $\mu_0=1$. Taking the derivative gives
\[   n z^{n-1} = \sum_{k=1}^n c_{k,n} \Phi_k'(z), \]
and using \eqref{Phider} then leads to
\[    n z^{n-1} = \sum_{k=1}^n c_{k,n} [k\Phi_{k-1}(z) + B_k \Phi_{k-2}(z)].  \]
Integrating then gives
\[     n \mu_{n-1} = c_{1,n} + c_{2,n} B_2.  \]
The relation \eqref{ckn} then given
\[  c_{1,n} = \kappa_1^2 \int_0^{2\pi} \overline{z-a_0} z^n \, d\nu(\theta) = \kappa_1^2 (\mu_{n-1} - a_0 \mu_n), \]
and
\[   c_{2,n} = \kappa_2^2 \int_0^{2\pi} \overline{z^2-a_0(1-a_1)z-a_1} z^n\, d\nu(\theta) = 
  \kappa_2^2[\mu_{n-2}-a_0(1-a_1)\mu_{n-1} -a_1\mu_n].  \]
If we use $\kappa_2^2=\kappa_1^2/(1-a_1^2)$ and $\kappa_1^2=1/(1-a_0^2)$, then we have
\[ n\mu_{n-1} = \frac{1}{1-a_0^2} \left( \mu_{n-1} -a_0 \mu_n +\frac{t}{2}(1-a_0^2)[\mu_{n-2}-a_0(1-a_1)\mu_{n-1}-a_1\mu_n] \right).  \]
The relation \eqref{dPrec} with $n=0$ and $a_{-1}=-1$ gives
\[  a_0 = \frac{t}{2}(1-a_0^2)(1-a_1), \quad \frac{t}{2}(1-a_0^2)a_1 = \frac{t}{2}(1-a_0^2) - a_0, \]
and if we use this to remove $a_1$ from the equation, then
\begin{eqnarray*}
   n\mu_{n-1} &=& \frac{1}{1-a_0^2} \left( \mu_{n-1} - a_0 \mu_n + \frac{t}{2}(1-a_0^2) \mu_{n-2}
 - a_0^2\mu_{n-1} -\frac{t}{2} (1-a_0^2)\mu_n + a_0\mu_n \right) \\
   &=& \mu_{n-1}- \frac{t}{2} \mu_n + \frac{t}{2} \mu_{n-2},
\end{eqnarray*}
which is the recurrence relation in \eqref{murec}.
\end{proof}

\section{Proof of Theorem \ref{thm1}}
By means of Lemmas \ref{lem1} and \ref{lem2} we have now established that every solution of the discrete
Painlev\'e II equation \eqref{dPrec} with $a_{-1}=-1$ and $-1 < a_n < 1$ for every $n \geq 0$, corresponds
to the Verblunsky coefficients of orthogonal polynomials on the unit circle, for which the moments satisfy the linear
recurrence relation
\[    (n-1)\mu_{n-1} = \frac{t}{2} (\mu_{n-2}-\mu_n).   \]
The general solution of this recurrence relation is a linear combination of two linearly independent solutions.
A well known result for modified Bessel functions is that they satisfy the recurrence
\[    Z_{\nu-1}(z)-Z_{\nu+1}(z) = \frac{2\nu}{z} Z_{\nu}(z), \]
i.e., both $Z_\nu(z) = I_\nu(z)$ and $Z_\nu(z) = e^{\nu\pi i} K_\nu(z)$ are a solution of this recurrence,
see e.g., \cite{NIST} or \cite[Eq. 10.29.1]{DLMF}. This means that
\[   \mu_n = C_1 I_n(t) + C_2 (-1)^n K_n(t)  , \]
for constants $C_1$ and $C_2$. For $t$ fixed one has as $n \to \infty$ 
\[   I_n(t) \sim \frac{1}{\sqrt{2\pi n}} \left( \frac{et}{2n} \right)^n, \quad
    K_n(t) \sim \sqrt{\frac{\pi}{2n}} \left( \frac{et}{2n} \right)^{-n}, \]
 (see \cite[10.41.1 and 10.41.2]{DLMF})
so that $I_n(t)$ remains bounded as $n \to \infty$ and $K_n$ is unbounded. Since
\[  | \mu_n| = \left| \int_0^{2\pi} z^n \, d\nu(\theta) \right| \leq \int_0^{2\pi} |z^n|\, d\nu(\theta) = \mu_0, \]
it follows that we have to take $C_2=0$, so that
\[   \mu_n = C_1 I_n(t).  \]
For $n=0$ we have $\mu_0=1$ giving $C_1=1/I_0(t)$, and hence 
\[   \mu_n = \frac{I_n(t)}{I_0(t)}.  \]
There is only one measure on the unit circle with these moments. The integral representation
\[    I_n(t) = \frac{1}{\pi} \int_0^\pi e^{t\cos \theta} \cos n\theta\, d\theta, \]
(see \cite[Eq. 10.32.3]{DLMF}) shows that the measure $\nu$ is absolutely continuous with density $v(\theta) =e^{t\cos \theta}/[2\pi I_0(t)]$. This is the measure that appears in the random unitary matrix model that was investigated by Periwal and Shevitz
\cite{PeriwalShevitz}, for which they showed that the Verblunsky coefficients satisfy \eqref{dPrec}. We have now shown that this is the only measure with this property, and hence there is a unique solution of \eqref{dPrec} with $a_{-1}=-1$ and $-1 < a_n < 1$ for every $n \geq 0$.

\section{Bounds for the unique solution}   \label{sec4}
We return to the discrete Painlev\'e equation as given by \eqref{dPII}. An upper bound for the
unique solution described in Theorem \ref{thm1} is given by:

\begin{theorem}  \label{thm2}
Let $\alpha \neq 0$, then the solution of \eqref{dPII} which satisfies $x_0=1$
and $-1 < x_n < 1$ for $n \geq 1$ satisfies
\begin{equation}  \label{bound}
   |x_n| \leq \frac{2^n}{|\alpha|^n} \frac{1}{(2n-1)!!} = \frac{4^n}{|\alpha|^n} \frac{n!}{(2n)!}, \qquad n \geq 0.
\end{equation}
\end{theorem}

\begin{proof}
We will make some successive approximations and start with
\begin{equation}  \label{b0}
     |x_n| \leq 1 := b_{n,0}, \qquad n \geq 0. 
\end{equation}
From Theorem \ref{thm1} we know that $-1 < x_n < 1$ for $n \geq 1$, hence from \eqref{dPII}
we find
\[     |\alpha| n|x_n| = (1-x_n^2)|x_{n+1} + x_{n-1}| \leq |x_{n+1}| + |x_{n-1}| \]
so that
\begin{equation}  \label{upbound}
     |x_n| \leq \frac{|x_{n+1}|+ |x_{n-1}|}{|\alpha| n}, \qquad n \geq 1.
\end{equation}
If we insert our initial bound (\ref{b0}) in this inequality, then
\begin{equation}  \label{b1}
     |x_n| \leq \frac{2}{|\alpha| n} := b_{n,1}, \qquad n \geq 1.      
\end{equation}
In general, our $k$-th approximation is $|x_n| \leq b_{n,k}$ for $n \geq k$, and \eqref{upbound}
gives the recurrence relation
\begin{equation}  \label{bk}
       b_{n,k} = \frac{b_{n+1,k-1} + b_{n-1,k-1}}{|\alpha| n}, \qquad n \geq k.  
\end{equation}
We claim that
\begin{equation} \label{bnk}
    b_{n,k} = \frac{2^k}{|\alpha|^k} \frac{(n-k-1)!!}{(n+k-1)!!}, \qquad n \geq k
\end{equation}
where
\[     n!! = \begin{cases} n (n-2)(n-4)\cdots 2, & \textrm{if $n$ is even}, \\
                           n (n-2) (n-4) \cdots 1, & \textrm{if $n$ is odd}.
             \end{cases}  \]
This can easily be proved by induction on $k$. Indeed, for $k=0$ the formula \eqref{bnk} gives
$b_{n,0} = 1$, which corresponds to \eqref{b0}. For $k=1$ formula \eqref{bnk} gives
$b_{n,1} = 2/(|\alpha| n)$ for $n \geq 1$, which corresponds to \eqref{b1}. Suppose that
\eqref{bnk} is true for $k-1$, then \eqref{bk} gives for $n \geq k$
\begin{eqnarray*}     b_{n,k} &=& \frac{1}{|\alpha| n} \left( \frac{2^{k-1}}{|\alpha|^{k-1}}
                                \frac{(n-k+1)!!}{(n+k-1)!!} + \frac{2^{k-1}}{|\alpha|^{k-1}}
                                \frac{(n-k-1)!!}{(n+k-3)!!} \right) \\
                     & = & \frac{1}{|\alpha| n} \frac{2^{k-1}}{|\alpha|^{k-1}} \frac{(n-k-1)!!}{(n+k-1)!!}
                           \left( n-k+1 + n+k-1 \right)
\end{eqnarray*}
and the right hand side is indeed equal to \eqref{bnk}. We now have the bound
\[    |x_n| \leq b_{n,k} = \frac{2^k}{|\alpha|^k} \frac{(n-k-1)!!}{(n+k-1)!!}, \qquad n \geq k. \]
The bound in \eqref{bound} corresponds to the choice $k=n$.
\end{proof}

Stirling's formula gives the asymptotic behavior
\[   \frac{4^n}{|\alpha|^n} \frac{n!}{(2n)!} \sim \frac{1}{\sqrt{2}} \left( \frac{e}{|\alpha| n} \right)^n \]
which shows that $x_n$ tends to zero faster than exponentially.

\begin{verbatim}
Walter Van Assche
Department of Mathematics, KU Leuven
Celestijnenlaan 200B box 2400
BE-3001 Leuven, Belgium
walter.vanassche@kuleuven.be
\end{verbatim}

\end{document}